\documentclass[11pt]{article}

\usepackage[dvips]{graphicx}
\usepackage{amssymb}
\usepackage{amsmath}
\usepackage{amsthm}
\usepackage{latexsym}
\usepackage[english]{babel}
\usepackage{appendix}
\usepackage{mathtools}

\usepackage[usenames,dvipsnames]{color}

\usepackage{hyperref}

\newtheorem{theorem}{Theorem}[section]

\newtheorem{lemma}{Lemma}[section]
\newtheorem{proposition}{Proposition}[section]

\newtheorem{corollary}{Corollary}[section]

\def \C {\mathbb{C}}
\def \R {\mathbb{R}}

\def \S {\mathbb{S}}

\def \p {\partial}

\begin{document}

\title{Singular Solutions for the Conformal Dirac-Einstein Problem on the Sphere}

\author{ Ali Maalaoui$^{(1)}$ \& Vittorio Martino$^{(2)}$  \& Tian Xu$^{(3)}$}
\addtocounter{footnote}{1}
\footnotetext{Department of Mathematics, Clark University, 950 Main Street, Worcester, MA 01610, USA. E-mail address: 
{\tt{amaalaoui@clarku.edu}}}
\addtocounter{footnote}{2}
\footnotetext{Dipartimento di Matematica, Alma Mater Studiorum - Universit\`a di Bologna. E-mail address:
{\tt{vittorio.martino3@unibo.it}}}
\addtocounter{footnote}{3}
\footnotetext{Department of Mathematics, Zhejiang Normal University, Jinhua, Zhejiang, 321004, China. E-mail address:
{\tt{xutian@amss.ac.cn}}}

\date{}
\maketitle

\vspace{5mm}

{\noindent\bf Abstract} {\small In this paper we investigate the existence of singular solutions to the conformal Dirac-Einstein system. Because of its conformal invariance, there are many similarities with the classical construction of singular solutions for the Yamabe problem. We construct here a family of singular solutions, on the three dimensional sphere, having exactly two singularities.
}

\vspace{5mm}

\noindent
{\small Keywords: Dirac-Einstein problem; Singular solutions.} 

\vspace{4mm}

\noindent
{\small 2010 MSC. Primary: 53C27.  Secondary: 35R01.}

\vspace{4mm}


\section{Introduction and main results}

\noindent
The conformal Dirac-Einstein system on a three dimensional spin manifold $(M,g,\Sigma_{g}M)$ consists of two equations:

\begin{equation}\label{eq1}
\left\{\begin{array}{ll}
L_{g} u=|\psi|^{2}u\\
& \text{on } M ,\\
D_{g}\psi=|u|^{2}\psi
\end{array}
\right.
\end{equation}
where $L_{g}=-\Delta_{g}+\frac{R_{g}}{8}$ is the conformal Laplacian and $D_{g}$ is the Dirac operator (we refer to the next section for further details). This system corresponds to the Euler-Lagrange equation of the energy functional
$$E:H^{1}(M)\times H^{\frac{1}{2}}(\Sigma_{g}M)\to \R ,$$
$$ E(u,\psi)=\int_{M}uL_{g}u+\langle D_{g}\psi,\psi \rangle \ dv_{g}-\int_{M}|u|^{2}|\psi|^{2}\ dv_{g},$$ 
here $\langle \cdot, \cdot \rangle$ denotes the compatible Hermitian metric on $\Sigma_{g}M$. This energy functional is derived from the total Dirac-Einstein energy functional 
$$\mathcal{E}(g,\psi)=\int_{M}R_{g}+\langle D_{g} \psi, \psi \rangle - |\psi|^{2}\ dv_{g},$$
by restricting the variations to a fixed conformal class of the metric. There have been many investigations of this functional since it is an extension of the classical Einstein-Hilbert functional. For instance, we refer the reader to \cite{FSY2,KF,MM2022}. The conformal version of the problem was also investigated in \cite{MM2019, GMM} on compact manifolds and in \cite{BMM} on manifolds with boundaries. We also point out that, in the two dimensional case, the conformal version of the problem leads to the super-Liouville equation which was investigated in \cite{J1,J2}. As in the case of the Yamabe problem \cite{MP1,MP2,MS, Schoen1988, SY1988}, the $Q$-curvature problem \cite{ACDRS,HS}, the CR Yamabe problem \cite{GMM2} and the Spinorial Yamabe problem \cite{MST}, a natural question that arises is the existence of singular solutions to $(\ref{eq1})$.

\noindent
In this paper, we propose to construct singular solutions for $(\ref{eq1})$ on the standard three sphere. As in the classical case, we believe that these solutions are the building blocks of singular solutions on general manifolds by means of a gluing construction, but this is beyond the scope of our investigation for now.
So let $(\S^3, g_s, \Sigma \S^3)$ be the unit sphere of dimension three equipped with its standard metric $g_s$ and its canonical spin bundle $\Sigma \S^3$. We are interested in finding singular solutions of (\ref{eq1}) on $\S^3\setminus \Lambda$, where $\Lambda$ is a pair of antipodal points on $\S^3$, that is we are looking for solutions of
\begin{equation}\label{eqspheresingular}
\left\{\begin{array}{ll}
L_{g_s} u=|\psi|^{2}u\\
& \text{on } \S^3\setminus \Lambda ,\\
D_{g_s}\psi=|u|^{2}\psi
\end{array}
\right.
\end{equation}
in the distributional sense, with $u$ and $\psi$ singular on $\Lambda$. This can be seen as a coupled singular Yamabe and spinorial Yamabe problem. We will assume for now that $\Lambda=\{N,S\}$, namely the north and south poles of the sphere $\S^{3}$.

\noindent
Again, by means of the stereographic projection, the system (\ref{eqspheresingular}) turns into
\begin{equation}\label{eqR3singular}
\left\{\begin{array}{ll}
-\Delta u=|\psi|^{2}u\\
& \text{on } \R^3\setminus \{0\} ,\\
D\psi=|u|^{2}\psi
\end{array}
\right.
\end{equation}

\noindent
One can approach the previous equations in two different ways, which however lead to the same system. The first one is more geometric and it starts by noticing that $\R^{3}\setminus \{0\}$ is conformal to $\R \times \S^{2}$ via the conformal map 
$$\big(r,\theta\big)\to \big(-\ln(r),\theta\big)=\big(t,\theta\big) .$$ 
Now, using the conformal invariance of the Laplacian and the Dirac operator, the problem (\ref{eqR3singular}) becomes:

\begin{equation}\label{eqRtimesS2}
\left\{\begin{array}{ll}
L_{g_{prod}} w =|\phi|^{2} w\\
& \text{on } \R \times \S^{2} ,\\
D_{g_{prod}} \phi= w^{2}\phi
\end{array}
\right.
\end{equation}
where $L_{g_{prod}}=-\Delta_{t,\theta}+\frac{1}{4}$ and $D_{g_{prod}}$ are the conformal Laplacian and the Dirac operator on $\R \times \S^{2}$ equipped with the canonical product metric $g_{prod}=dt^2 + d\theta^2$. We recall from \cite{BX, SX2020} that the Spin bundle over $\R\times \S^{2}$ can be identified with
$$\Sigma_{g_{prod}}\Big(\R\times \S^{2}\Big)=\Big(\Sigma_{dt^{2}}\R\oplus \Sigma_{dt^{2}}\R\Big) \otimes \Sigma_{d\theta^{2}}\S^{2}$$
and the Dirac operator is given by
$$D_{g_{prod}}\Big((\psi_{1}\oplus \psi_{2})\otimes \varphi\Big)=\Big(D_{dt^{2}}\psi_{1}\oplus -D_{dt^{2}}\psi_{2}\Big)\otimes \varphi + \omega_{\C}\cdot_{dt^{2}}(\psi_{1}\oplus \psi_{2})\otimes D_{d\theta^{2}}\varphi,$$
where $\psi_{1},\psi_{2} \in C^{\infty}(\R,\Sigma_{dt^{2}}\R)$, $\varphi \in C^{\infty}(\S^{2},\Sigma_{d\theta^{2}}\S^{2})$, $\omega_{\C}$ is the Chirality operator on $\R$ and "$\cdot_{dt^{2}}$" is the Clifford multiplication on $\Sigma_{dt^{2}}\R$. Choosing $\varphi$ to be a Killing spinor on $\S^{2}$, we have that $\varphi$ is an eigenspinor for the Dirac operator on $\S^{2}$ with eigenvalue $1$ and since $|\varphi|$ is constant, we can assume it to be $1$. For the scalar part $w$, we will look for solution of the form $w(t,\theta)=u(t)$. Then, by taking into account the natural splitting of the spin bundle, one gets the following system for $\psi=\psi^{+}\oplus \psi^{-}$:
\begin{equation}\label{equationupsit}
\left\{\begin{array}{lll}
-u''+\frac{1}{4}u=(|\psi^{+}|^{2}+|\psi^{-}|^{2})u\\
\\
\displaystyle i\frac{d\psi^{+}}{dt}+i\psi^{-}=u^2 \psi^{+}\\
\\
\displaystyle - i\frac{d\psi^{-}}{dt}-i\psi^{+}=u^2\psi^{-}
\end{array}
\right.
\end{equation}
Since $\psi$ is a complex valued function, if we put $\psi^{+}=a+ib$ and $\psi^{-}=a-ib$ in the previous system, we get the following system:
\begin{equation}\label{systemab}
\left\{\begin{array}{ll}
u''=-(a^{2}+b^{2})u+\frac{1}{4}u\\
a'=-a +u^{2}b\\
b' =b- u^{2}a.
\end{array}
\right.
\end{equation}


\noindent
The second approach is more analytical in nature. It was used in several works as an ansatz to find particular solutions for equations involving the Dirac operator (we refer the reader to \cite{FSY2, Rot, Soler, MST} and the references therein). We start by defining the space of ``radial" spinors $E(\R^{3})$ as follows:
\begin{equation}\label{ER3}
E(\R^{3})=\left\{\psi(x)=f_{1}(|x|)\gamma_{0}+\frac{f_{2}(|x|)}{|x|}x\cdot \gamma_{0}\; ; x\in \R^{3}, f_{1},f_{2}\in C^{\infty}(0,\infty), \gamma_{0}\in \S^{2}_{\C}\right\},
\end{equation}
where "$\cdot$" stands for the Clifford multiplication and $\S^{2}_{\C}$ denotes the complex unit sphere in $\C^2$; we notice that this space is stable under the action of the Dirac operator. This second approach relies on the ansatz that $u(x)=u(|x|)$ and $\psi \in E(\R^{3})$: so, if in (\ref{eqR3singular}) we apply the Emden-Fowler change of variable $r=e^{-t}$ and write $f_{1}(r)=-a(t)e^{t}$, $f_{2}(r)=b(t)e^{t}$, we obtain again the system $(\ref{systemab})$.

\noindent
Now, if we introduce the variable $v:=u'$, we see that the system $(\ref{systemab})$ can be viewed as a first order Hamiltonian system
\begin{equation}\label{ham}
\left\{\begin{array}{ll}
u'=v\\
v'=-\left(a^{2}+b^{2}-\frac{1}{4}\right)u\\
a'=-a +u^{2}b\\
b' =b- u^{2}a
\end{array}
\right.
\end{equation}
where the Hamiltonian function $H$ is given by:
\begin{align}\label{HamiltonianH}
H(u,v,a,b)&=\frac{v^{2}}{2}+\frac{u^{2}}{2}\left(a^{2}+b^{2}-\frac{1}{4} \right)-ab\notag\\
&=\frac{v^{2}}{2}+\frac{(u^{2}-1)}{2}\left(a^{2}+b^{2}-\frac{1}{4}\right)+\frac{1}{2}\left((a-b)^{2}-\frac{1}{4}\right) . \notag
\end{align}
Hence,
\begin{equation*}
\left\{\begin{array}{ll}
\dot u = \displaystyle \frac{\p H}{\p v} (u,v,a,b)\\
\\
\dot v =-\displaystyle \frac{\p H}{\p u} (u,v,a,b)\\
\\
\dot a = \displaystyle \frac{\p H}{\p b} (u,v,a,b)\\
\\
\dot b =-\displaystyle \frac{\p H}{\p a} (u,v,a,b).
\end{array}
\right.
\end{equation*}

\noindent
The equilibrium points of this system, with $u\geq0$,  are
$$P_0=(0,0,0,0), \qquad P_{\pm}=\left( 1,0, \pm\frac{1}{2\sqrt{2}},\pm\frac{1}{2\sqrt{2}} \right),$$
and they correspond to the energy levels $H=0$ and $H=-\frac{1}{8}$; in particular $P_0$ is a saddle point and $P^{\pm}$ are center points. This structure is similar to the case of the Spinorial Yamabe in \cite{MST}. From the analysis of this Hamiltonian system, we have the following result:
\begin{theorem}\label{maintheorem}
Let $T_0=2^{\frac{3}{4}}\pi$. Then there exist $T_{0}<T_{1}\leq T_{2}$ such that for $T\in (T_0,T_{1})\cup (T_{2},+\infty)$, there exists a family $(u_{T},v_{T},a_{T},b_{T})$ of non-constant $2T$-periodic solutions to $(\ref{ham})$. Moreover,
\begin{itemize}
\item[(i)] when $T\to T_0$,
            $$\left(u_{T},v_{T},a_{T},b_{T}\right)\to \left(1,0,\frac{1}{2\sqrt{2}},\frac{1}{2\sqrt{2}} \right),$$
in  $C^{2,\alpha}_{loc}(\R) \times C^{2,\alpha}_{loc}(\R)\times  C^{1,\beta}_{loc}(\R) \times  C^{1,\beta}_{loc}(\R), \; 0 < \alpha, \beta < 1 ;$
\item[(ii)] when $T\to \infty$, there exists $t_{0}\in \R$ such that
             $$\left( u_{T},v_{T}, a_{T}, b_{T} \right) \to \left( u_{0}(\cdot-t_{0}),u'_{0}(\cdot-t_{0}), a_{0}(\cdot-t_{0}), b_{0}(\cdot-t_{0}) \right),$$ 
in $ C^{2,\alpha}_{loc}(\R) \times C^{2,\alpha}_{loc}(\R)\times  C^{1,\beta}_{loc}(\R) \times  C^{1,\beta}_{loc}(\R),\; 0 < \alpha, \beta < 1 ,$
where 
$$\left(u_0(t),a_0(t),b_{0}(t)\right):=\Big( 2^{-\frac{1}{4}}\cosh^{-\frac{1}{2}}(t), \frac{3}{2\sqrt{2}}e^{-\frac{t}{2}}\cosh^{-\frac{3}{2}}(t), \frac{3}{2\sqrt{2}}e^{\frac{t}{2}}\cosh^{-\frac{3}{2}}(t)\Big)$$
is the solution of (\ref{ham}) given by the nontrivial homoclinic orbit.
 \end{itemize}
\end{theorem}


\noindent
In terms of singular solutions of the conformal Dirac-Einstein equation (\ref{eqR3singular}) on $\R^3\setminus \{0\}$, which is equivalent to the problem (\ref{eqspheresingular}) on $\S^{3}\setminus \Lambda$, we obtain
\begin{corollary}\label{Corollary1}
For $T>T_{2}>0$, there exist $\lambda>0$, $\Phi_{0}\in \S_{\C}^{2}$ and a one parameter family $(u_{T},\psi_{T})$ of singular solutions of problem (\ref{eqR3singular}) such that, when $T\to \infty$
$$\left( u_{T}, \psi_{T} \right) \to \left( U_{\lambda}, \Psi_{\lambda} \right), \;  \text{ in  } C^{2,\alpha}_{loc}(\R^{3}\setminus \{0\}) \times C^{1,\beta}_{loc}(\Sigma\R^{3}\setminus \{0\}), \; 0 < \alpha, \beta < 1 ,$$
where $\big(U_{\lambda}(x),\Psi_{\lambda}(x)\big):=\left(\Big(\frac{2\lambda}{\lambda^{2}+|x|^{2}}\Big)^{\frac{1}{2}}, \Big(\frac{2\lambda}{\lambda^{2}+|x|^{2}}\Big)^{\frac{3}{2}}(1- x)\cdot\Phi_{0}\right)$ .
\end{corollary}

\section{Geometric and Analytical Settings}

\noindent
Here we briefly recall some notations and properties of the relevant operators that we are going to use. On a general compact, without boundary, three dimensional Riemannian manifold $(M,g)$, we consider the conformal Laplacian acting on functions by
$$L_g u:=-\Delta_g u+\frac{1}{8}R_g u ,$$
where $\Delta_g$ is the standard Laplace-Beltrami operator and $R_g$ is the scalar curvature. Then the conformal invariance of $L_g$ means that if $\tilde g=g_u= u^4g$ is a metric in the conformal class of $g$, then we have $L_{\tilde g}f=u^{-5}L_g(uf) .$ We will denote by $H^1(M)$ the usual Sobolev space on $M$, and we recall that by the Sobolev embedding theorems there is a continuous embedding $$H^1(M)\hookrightarrow L^p(M), \quad 1\leq p \leq 6 ,$$
which is compact if $1\leq p <6$.\\
Regarding the spinorial part, we denote by $\Sigma M$ the canonical spinor bundle associated to $M$ (see for instance \cite{Fred}), whose sections are called spinors. On this bundle one defines a natural Clifford multiplication
$$\text{Cliff}:C^\infty(TM\otimes\Sigma M)\longrightarrow C^\infty(\Sigma M), $$
a hermitian metric $\langle \cdot, \cdot \rangle$,  and a natural metric connection
$$\nabla^\Sigma:C^\infty(\Sigma M)\longrightarrow C^\infty(T^*M\otimes\Sigma M) .$$
Therefore the Dirac operator $D_g$ acting on spinors is given by the composition
$$D_g:C^\infty(\Sigma M)\longrightarrow C^\infty(\Sigma M), \quad D_g=\text{Cliff} \circ \nabla^\Sigma ,$$
where $T^*M \simeq TM$ are identified by means of the metric; the conformal invariance for the Dirac operator reads as follows: if $\tilde g= u^4g$, then $D_{\tilde g}\psi=u^{-4}D_g(u^2\psi) .$

\noindent
After the reduction defined introduced in $(\ref{systemab})$, it is natural to introduce the operator $A$ defined by 
\begin{equation}\label{operatorA}
Az=-Jz'+JBz,
\end{equation}
for any smooth complex valued 2T-periodic function 
$$z(t)=a(t) + i b(t) \simeq \left(\begin{array}{cc}
a(t)\\
b(t)
\end{array}
\right), $$
where
$$
J=\left[\begin{array}{cc}
0&1\\
-1&0
\end{array}
\right], \qquad
B=\left[\begin{array}{cc}
-1&0\\
0&1
\end{array}
\right]. $$
The natural domain for $A$ is $H^{\frac{1}{2}}_{per}(T) := H^{\frac{1}{2}}_{per}([-T,T]; \C)$. Moreover, since $JB=-BJ$, it is clear that $A: H^{\frac{1}{2}}_{per}(T)\to H^{-\frac{1}{2}}(T)$ is self-adjoint in $L^{2}$. Moreover, since  $A^{2}z=-z''+z$, $A$ has a trivial kernel. The operator $A$ has a compact resolvent and there exists a complete $L^{2}$-orthonormal basis of eigenfunctions $\{z_i\}_{i\in\mathbb{Z}}$ satisfying
$$A z_i=\lambda_i z_i ,$$
where the eigenvalues $\{\lambda_i\}_{i\in\mathbb{Z}}$ are unbounded, that is $|\lambda_i|\rightarrow\infty$, as $|i|\rightarrow\infty$. We refer the reader to \cite{SX2020} for more properties of operators of the same type as $A$.
For a given $z\in H^{\frac{1}{2}}_{per}(t)$, written as $z=\sum_{i\in \mathbb{Z}}\alpha_{i}z_{i}$, we define the operator
$$|A|^{s}: H_{per}^{\frac{1}{2}}(T)\rightarrow L^{2}([-T,T]) , \quad |A|^{s}(z)=\sum_{i\in \mathbb{Z}} \alpha_{i}|\lambda_{i}|^{s}z_{i}.$$
We can therefore define the inner product
$$\langle f,g\rangle_{s}=\langle|A|^{s}f,|A|^{s}g\rangle_{L^{2}},$$
which induces an equivalent norm in $H^{s}_{per}([-T,T],\C)$; in particular, for $s=\frac{1}{2}$, we will consider
$$\langle z,z\rangle_{\frac{1}{2}}=\|z\|^{2}_{\frac{1}{2}}.$$
That is, 
$$\|z\|^2_{\frac{1}{2}} = \|z\|^2_{H^{\frac{1}{2}}}=  \int_{-T}^{T} \left||A|^{\frac{1}{2}}z\right|^2 dt.$$
Now, given the $L^2$-orthonormal basis of eigenfunctions $\{z_i\}_{i\in\mathbb{Z}}$, we will denote by $z_i^-$ the eigenspinors with negative eigenvalue, $z_i^+$ the eigenspinors with positive eigenvalue. Therefore we set
$$H^{\frac{1}{2},-}:=\overline{\text{span}\{z_i^-\}_{i\in\mathbb{Z}}},\quad
H^{\frac{1}{2},+}:=\overline{\text{span}\{z_i^+\}_{i\in\mathbb{Z}}},$$
where the closure is with respect to the topology induced by the previous norm, and we have the splitting:
\begin{equation}\label{splitting}
H^{\frac{1}{2}}_{per}(T)=H^{\frac{1}{2},-}\oplus  H^{\frac{1}{2},+},
\end{equation}
and we will denote by $P^{+}$ and $P^{-}$ be the projectors on $H^{\frac{1}{2},+}$ and $H^{\frac{1}{2},-}$ respectively. We will also use the notation $z^{\pm}:=P^{\pm}z$ for all $z\in H_{per}^{\frac{1}{2}}(T)$. Finally, we recall a regularity results for weak solutions, (see \cite{MM2019}, Theorem 3.1): if $(u,\psi)\in H^1(M) \times H^{\frac{1}{2}}(\Sigma M)$ is a weak solution of the system of equations (\ref{eq1}), on a closed three dimensional spin manifold $(M,g, \Sigma M)$, then $(u,\psi)\in C^{2,\alpha}(M) \times C^{1,\beta}(\Sigma M)$, for some $0<\alpha,\beta<1$.


\section{Proof of the Main Theorem}

\noindent
In this section we will prove the Theorem \ref{maintheorem}. We will start by proving the existence of periodic orbits of (\ref{ham}) near the equilibrium points and then we will address the existence of periodic orbits with large periods.


\subsection{Small Oscillation and Periodic Orbits}

\noindent
In order to show the existence of periodic orbits near the equilibrium points $P_{\pm}$, we will use the Hamiltonian structure of $(\ref{ham})$ after rewriting it. First of all, we transform the system (\ref{ham}) using the auxiliary variables $\bar{a}=\frac{a+b}{\sqrt{2}}$ and $\bar{b}=\frac{a-b}{\sqrt{2}}$. So the system becomes:
\begin{equation}\label{hamnew}
\left\{\begin{array}{ll}
u'=v\\
v'=\left(\frac{1}{4}-(\bar{a}^{2}+\bar{b}^{2})\right)u\\
\bar{a}'=-(1+u^{2}) \bar{b} \\
\bar{b}' =(u^{2}-1) \bar{a}
\end{array}
\right.
\end{equation}

\noindent
The new Hamiltonian is then
$$\bar{H}(u,v,\bar{a},\bar{b})=\frac{1}{2}v^{2}+\frac{1}{2}(u^{2}-1)\left(\bar{a}^{2}+\bar{b}^{2}-\frac{1}{4}\right)+\frac{1}{2}\left(2\bar{b}^{2}-\frac{1}{4}\right).$$
The equilibrium points are $(0,0,0,0)$ and $(1,0,\pm\frac{1}{2},0)$. The linearization of the right hand side of (\ref{hamnew}) at $(1,0,\frac{1}{2},0)$ leads to the following matrix
$$C=\left[\begin{array}{cccc}
0&1&0&0\\
0&0&-1&0\\
0&0&0&-2\\
1&0&0&0
\end{array}
\right].$$

\noindent
In particular, $C$ has two real eigenvalues $\pm 2^{\frac{1}{4}}$ and two complex eigenvalues $\pm i2^{\frac{1}{4}}$. Therefore, one can apply the Lyapunov's center theorem, in order to exhibit the existence of a positive $\delta$ and a family $(x_{r})_{r\in (-\delta,\delta)}:=(u_{T_{r}},v_{T_{r}},\overline{a}_{T_{r}},\overline{b}_{T_{r}})$ of periodic solutions with a period $T_{r}$, starting from the equilibrium point $(1,0,\frac{1}{2},0)$. Moreover, the period $T_{r}$ converges to $T_{0}=\frac{2\pi}{2^{\frac{1}{4}}}$, when $r\to 0$; which proves the first part of Theorem \ref{maintheorem}.


\subsection{Solutions with Large Period}

We focus now on proving the existence of $2T$-periodic solutions with $T$ large; we will also show that these periodic solutions are different from the constant solution, proving the second part of the main theorem. The strategy here is different from the one used above. We will use a variational framework, keeping the equation on the scalar part $u$, while using the Hamiltonian structure of the second two equations in (\ref{systemab}).

\noindent
We will set $H^{1}_{per}(T)=H^{1}_{per}([-T,T]; \R)$, the Sobolev space of $2T$-periodic functions endowed with the equivalent norm
$$\|u\|_{1}^2 = \|u\|^2_{H^{1}}=\int_{-T}^{T}|u'|^{2}+\frac{1}{4}u^{2} dt . $$
We also denote by 
$$\|(u,z) \|^2=\|u\|_{1}^2 + \|z\|^2_{\frac{1}{2}}, \quad (u,z)\in H^{1}_{per}(T)\times H^{\frac{1}{2}}_{per}(T) $$
Now let us consider the functional
$$E:H^{1}_{per}(T)\times H^{\frac{1}{2}}_{per}(T)\rightarrow \R, $$
$$E(u,z)=\frac{1}{2}\int_{-T}^{T}|u'|^{2}+\frac{1}{4}u^{2}+\langle Az,z\rangle\ dt-\frac{1}{2}\int_{-T}^{T} u^{2}|z|^{2} \ dt,$$
where $A$ is the operator defined in \ref{operatorA}. A direct computation shows that critical points of $E$ solve the system (\ref{systemab}). Hence, to prove our result, we need to show the existence of critical points of the functional $E$. We start by showing a compactness property of our functional.

\begin{lemma}\label{EP-S}
The functional $E$ satisfies the Palais-Smale condition (PS).
\end{lemma}
\begin{proof}
The idea is very similar to the proof of the (PS) condition for the Dirac-Einstein equation as in \cite{MM2019}. So we consider a (PS) sequence $(u_{n},z_{n})\in H^{1}_{per}(T) \times H^{\frac{1}{2}}_{per}(T)$ at the level $c\in \R$. Therefore,
\begin{equation}\label{pse}
\int_{-T}^{T}|u'_{n}|^{2}+\frac{1}{4}u_{n}^{2}\ dt+\langle Az_{n},z_{n}\rangle dt-\int_{-T}^{T} u_{n}^{2}|z_{n}|^{2}\ dt\to 2c
\end{equation}
and
\begin{equation}\label{pss}
\left\{\begin{array}{ll}
-u''_{n}+\frac{1}{4}u_{n}=u_{n}|z_{n}|^{2}+o_{H^{-1}}(1)\\
\\
Az_{n}=u_{n}^{2}z_{n}+o_{H^{-\frac{1}{2}}}(1)
\end{array}
\right.
\end{equation}
Multiplying the first equation of $(\ref{pss})$ by $u_{n}$ and the second equation of $(\ref{pss})$ by $z_{n}$ and substituting it in $(\ref{pse})$ we have
$$\|u_{n}\|_{1}^{2}=2c+o(\|z_{n}\|_{\frac{1}{2}})$$
and
$$\int_{[-T,T]}u_{n}^{2}|z_{n}|^{2}\ dt = 2c+o(\|u_{n}\|_{1}+\|z_{n}\|_{\frac{1}{2}}).$$

\noindent
Moreover, we have that
\begin{align}
\|z_{n}^{+}\|^{2}_{\frac{1}{2}}&=\int_{[-T,T]}u_{n}^{2}\langle z_{n},z_{n}^{+}\rangle\ dt +o(\|z_{n}\|_{\frac{1}{2}})\notag\\
&\leq \left(\int_{[-T,T]}u_{n}^{2}|z_{n}|^{2} dt \right)^{\frac{1}{2}}\left(\int_{[-T,T]}u_{n}^{2}|z_{n}^{+}|^{2}dt\right)^{\frac{1}{2}}+o(\|z_{n}\|)\\
&\leq \big(2c+o(\|u_{n}\|_{1}+\|z_{n}\|_{\frac{1}{2}})\big)^{\frac{1}{2}}\|u_{n}\|_{L^{4}}\|z_{n}^{+}\|_{L^{4}}+o(\|z_{n}\|_{\frac{1}{2}})\notag\\
&\leq C\big(2c+o(\|u_{n}\|_{1}+\|z_{n}\|_{\frac{1}{2}})\big)\|z_{n}^{+}\|_{\frac{1}{2}} +o(\|z_{n}\|_{\frac{1}{2}}) . \notag
\end{align}
Similarly,
$$\|z_{n}^{-}\|^{2}_{\frac{1}{2}} \leq C\Big(2c+o(\|u_{n}\|_{1}+\|z_{n}\|_{\frac{1}{2}})\Big)\|z_{n}^{+}\|_{\frac{1}{2}}+o(\|z_{n}\|_{\frac{1}{2}}) . $$
Thus, $\|z_{n}\|_{\frac{1}{2}}$  and $\|u_{n}\|_{1}$ are bounded. So up to a subsequence, $u_{n}\rightharpoonup u_{\infty}$ weakly in $H^{1}_{per}(T)$ and strongly in $ C^{0,\frac{1}{2}}([-T,T])$. Moreover, $z_{n}\rightharpoonup z_{\infty}$ weakly in $H^{\frac{1}{2}}_{per}(T)$ and strongly in $L^{p}([-T,T])$ for all $1\leq p<\infty$. Hence, in order to finish the proof, we notice that
$$\|u_{n}\|_{1}^{2}=\int_{[-T,T]}u_{n}^{2}|z_{n}|^{2}\ dt+o(1).$$
But since $u_{n}\to u_{\infty}$ in $L^{\infty}([-T,T])$ and $z_{n}\to z_{\infty}$ in $L^{2}([-T,T])$, we see that $(u_{n})$ converges strongly to $u_{\infty}$ in $H^{1}_{per}(T)$ and a similar argument works for $(z_{n})$ which finishes the proof of the (PS) condition for $E$.
\end{proof}


\noindent
Our next step now is to show that $E$ has a mountain-pass geometry around zero, but this requires a reduction that compensates the strongly indefinite aspect of the functional. We start by the following 
\begin{proposition}\label{existenc of g}
Let us consider the natural splitting $H^{\frac{1}{2}}_{per}(T)=H^{\frac{1}{2},-} \oplus H^{\frac{1}{2},+}$ (as in \ref{splitting}). Then there exists a functional $g:H^{1}_{per}(T)\times H^{\frac{1}{2},+}\to H^{\frac{1}{2},-}$ satisfying, for $v\in H^{\frac{1}{2},+}$
\begin{equation}\label{eqe}
E(u,v+w)<E(u,v+g(u,v)), \text{ for all } w\in H^{\frac{1}{2},-}, w\not= g(u,v).
\end{equation}
\end{proposition}

\begin{proof}
We first notice that
\begin{align}
E(u,v+w)&=\frac{1}{2}\Big(\|u\|_{1}^{2}+\|v\|^{2}_{\frac{1}{2}}-\|w\|^{2}_{\frac{1}{2}}-\int_{-T}^{T} u^{2}|v+w|^{2}\ dt\Big)\notag\\
&=\frac{1}{2}\left(\|u\|_{1}^{2}+\|v\|^{2}_{\frac{1}{2}}-\int_{-T}^{T}u^{2}|v|^{2} dt \right)  + K(w),  \notag
\end{align}
where $K:H^{\frac{1}{2},-}\to \R$ is defined by
$$K(w)=-\|w\|^{2}_{\frac{1}{2}}-\int_{-T}^{T} u^{2}|w|^{2}\ dt-2\int_{-T}^{T} u^{2}\langle v,w\rangle\ dt$$
is strictly concave and anti-coercive. Therefore, it has a unique maximizer $w_{0}=g(u,v)\in H^{\frac{1}{2},+}$. This maximizer satisfies the equation 
\begin{equation}\label{pro}
Aw_{0}=P^{-}(u^{2}(w_{0}+v)),
\end{equation}
where $P^{-}$ the projector on $H^{\frac{1}{2},-}$. Thus, property $(\ref{eqe})$ is now satisfied. 
\end{proof}

\begin{lemma}
Let $(u,v) \in H^{1}_{per}(T)\times H^{\frac{1}{2},+}$ and $g$, the functional given by the previous Proposition \ref{existenc of g}. Let us define $F(u,v)=E(u,v+g(u,v))$. Then $F$ has the mountain pass geometry. Namely, we have
\begin{itemize}
\item[(i)] $F(0)=0$
\item[(ii)] There exists $r>0$ such that if $\|u\|_{1}^{2}+\|v\|^{2}_{\frac{1}{2}}\leq r$, then $F(u,v)\geq 0$; in particular if $\|u\|_{1}^{2}+\|v\|^{2}_{\frac{1}{2}}=r$, then $F(u,v)\geq \alpha=\alpha(r)>0$.
\item[(iii)] If $\int_{-T}^{T}u^{2}|v|^{2}\ dt \not=0$, there exist $t, s>0$ large enough, such that $F(tu,sv)<0$.
\item[(iv)] The functional $F$ satisfies the (PS) condition.
\end{itemize}
\end{lemma}
\begin{proof}
Regarding $(i)$, we notice that $g(0,0)=0$, hence $F(0)=0$. Next, we notice that
\begin{align}
F(u,v)&\geq \frac{1}{2}\left(\|u\|_{1}^{2}+\|v\|^{2}_{\frac{1}{2}}-\int_{-T}^{T}u^{2}|v|^{2}\ dt\right) \notag\\
&\geq\frac{1}{2}\left(\|u\|_{1}^{2}+\|v\|^{2}_{\frac{1}{2}}-\frac{1}{2}\left(\|u\|_{L^{4}}^{4}+\|v\|_{L^{4}}^{4}\right)\right)\notag\\
&\geq \frac{1}{2}\left(\|u\|_{1}^{2}+\|v\|^{2}_{\frac{1}{2}}-C\left(\|u\|_{1}^{4}+\|v\|^{4}_{\frac{1}{2}}\right)\right)\notag,
\end{align}
where we used in the second inequality the identity $2ab\leq a^{2}+b^{2}$ and in the second inequality, the classical Sobolev embedding. Hence, $(ii)$ is satisfied.\\
Now, we consider $u\in H^{1}_{per}(T)$ and $v\in H^{\frac{1}{2},+}$ such that $u|v|\not=0$ and we fix $\|v\|_{\frac{1}{2}}=1$. Let $t_n$ be an increasing divergent sequence. Then two possible cases can occur: 
$$\text{either } \frac{\|g(t_{n}u,t_{n}v)\|_{\frac{1}{2}}}{t_{n}}\to \infty\quad \text{ or }\quad \frac{\|g(t_{n}u,t_{n}v)\|_{\frac{1}{2}}}{t_{n}}\to a\geq 0.$$
In the first case, we have
\begin{align}
2F(t_{n}u,t_{n}v)&=t_{n}^{2}\|u\|_{1}^{2}+t^{2}_{n}\|v\|^{2}_{\frac{1}{2}}-\|g(t_{n}u,t_{n}v)\|^{2}_{\frac{1}{2}}-  t_{n}^{2}\int_{-T}^{T}u^{2}|t_{n}v+g(t_{n}u,t_{n}v)|^{2}\ dt\notag\\
&\leq t^{2}_{n}\left(\|u\|_{1}^{2}+\|v\|^{2}_{\frac{1}{2}}-\frac{\|g(t_{n}u,t_{n}v)\|^{2}_{\frac{1}{2}}}{t_{n}^{2}} \right)\to -\infty.\notag
\end{align}
In the second case, we let $h_{n}=t_{n}v+g(t_{n}u,t_{n}v)$ and we denote by $w$ the weak limit of $w_{n}=\frac{h_{n}}{\|h_{n}\|_{\frac{1}{2}}}$. We also notice that
$$\langle w_{n},v\rangle_{\frac{1}{2}}=\frac{t_{n}}{\|h_{n}\|_{\frac{1}{2}}}\to (1+a)^{-\frac{1}{2}} .$$
Therefore, $w=(1+a)^{-\frac{1}{2}}v+w^{-}$. Hence, we have
\begin{align}
2F(t_{n}u,t_{n}v)&=t_{n}^{2}\|u\|_{1}^{2}+t_{n}^{2}\|v\|^{2}_{\frac{1}{2}}-\|g(t_{n}u,t_{n}v)\|^{2}_{\frac{1}{2}}-t_{n}^{2}\int_{-T}^{T}u^{2}|h_{n}|^{2}\ dt \\
&=t_{n}^{2}\left(\|u\|_{1}^{2} + \|v\|^{2}_{\frac{1}{2}} -\frac{\|g(t_{n}u,t_{n}v)\|^{2}_{\frac{1}{2}}}{t_{n}^{2}}\right)-t_{n}^{4}\frac{\|h_{n}\|^{2}_{\frac{1}{2}}}{t_{n}^{2}}\int_{-T}^{T}u^{2}|w_{n}|^{2}\ dt.
\end{align}
But, $\int_{-T}^{T}u^{2}|w_{n}|^{2}\ dt \to \int_{-T}^{T}u^{2}|w|^{2}\ dt.$  Thus, in order to conclude, it is enough to show that  $ \int_{-T}^{T}u^{2}|w|^{2}\ dt\not=0$. For this end, we recall that $(\ref{pro})$ yields
$$-\|g(t_{n}u,t_{n}v)\|^{2}_{\frac{1}{2}}=t_{n}^{2}\int_{-T}^{T}u^{2}\langle t_{n}v+g(t_{n}u,t_{n}v),g(t_{n}u,t_{n}v)\rangle \ dt.$$
Hence,
\begin{align}
-\frac{\|g(t_{n}u,t_{n}v)\|^{2}_{\frac{1}{2}}}{t_{n}^{2}}&=t_{n}\|h_{n}\|_{\frac{1}{2}} \int_{-T}^{T}u^{2} \left \langle w_{n},\frac{g(t_{n}u,t_{n}v)}{t_{n}} \right\rangle \ dt\notag\\
&=t_{n}^{2}\frac{\|h_{n}\|_{\frac{1}{2}}}{t_{n}} \int_{-T}^{T}u^{2} \left \langle w_{n},\frac{g(t_{n}u,t_{n}v)}{t_{n}} \right \rangle \ dt.
\end{align}
But, $\frac{\|h_{n}\|_{\frac{1}{2}}}{t_{n}}\to (1+a)^{\frac{1}{2}}\not=0$ and $\frac{\|g(t_{n}u,t_{n}v)\|^{2}_{\frac{1}{2}}}{t_{n}^{2}}\to a^{2}$. Therefore, we see that  
$$\lim_{n\to \infty}\int_{-T}^{T}u^{2}\langle w_{n},\frac{g(t_{n}u,t_{n}v)}{t_{n}}\rangle \ dt=0.$$
On the other hand, $\frac{g(t_{n}u,t_{n}v)}{t_{n}}=\frac{\|h_{n}\|_{\frac{1}{2}}}{t_{n}}w_{n}-v$, converges weakly in $H^{\frac{1}{2}}_{per}(T)$ to $(1+a)^{\frac{1}{2}}w^{-}$ and thus strongly in $L^{2}([-T,T])$.
Therefore,
$$\int_{-T}^{T}u^{2}\langle w,w^{-}\rangle \ dt=0.$$
Hence, if $\int_{-T}^{T}u^{2}|v|^{2}\ dt \not=0$, then $ \int_{-T}^{T}u^{2}|w|^{2}\ dt\not=0$. Hence, $(iii)$ is satisfied.\\
Finally, in order to show that $F$ satisfies the (PS) condition, we first claim that $\|\nabla F(u,v)\|=\|\nabla E (u,v+g(u,v))\|$. Indeed, we recall that
$$\langle \nabla_{z}E(u,v+g(u,v)), w\rangle =0, \forall w\in H^{\frac{1}{2},-}.$$
Hence, for every $h\in H^{1}_{per}(T)$ we have
\begin{align}
\langle \nabla_{u}F(u,v),h\rangle & =\langle \nabla_{u}E(u,v+g(u,v)),h\rangle+\langle \nabla_{z}E(u,v+g(u,v)),\nabla_{u}g(u,v)\cdot h\rangle \\
&=\langle \nabla_{u}E(u,v+g(u,v)),h\rangle .\notag
\end{align}
Similarly, for all $w\in H^{\frac{1}{2},+}$ we have
\begin{align}
\langle \nabla_{v}F(u,v),w\rangle & =\langle \nabla_{z}E(u,v+g(u,v)),w+\nabla_{v}g(u,v)\cdot w\rangle\\
&=\langle \nabla_{z}E(u,v+g(u,v)),w\rangle\notag
\end{align}
and this proves the claim. Now, if $(u_{n},v_{n})$ is a (PS) sequence for $F$, then 
$$(u_{n},v_{n}+g(u_{n},v_{n}))$$
is a (PS) sequence for $E$ and using Lemma \ref{EP-S} we finish the proof of $(iv)$.
\end{proof}

\noindent
Using the mountain pass lemma, we know that $F$ has a critical point. But, as discussed above, we see that $\|\nabla F(u,v)\|=\|\nabla E(u,v+g(u,v))\|$. So the critical points of $F$ correspond to critical points of $E$.\\
One can characterize this critical point as the minimum of $E$ on the generalized Nehari manifold
$$\mathcal{N}=\left\{(u,z)\in H^{1}_{per}(T)\times H^{\frac{1}{2}}_{per}(T)\setminus \{(0,0)\}, \; \textit{satisfying } \; (*)
\right\}.$$
where
$$ (*) \quad
\left\{\begin{array}{ll}
\displaystyle\int_{-T}^{T}|u'|^{2}+\frac{1}{4}u^{2}dt=\int_{-T}^{T}u^{2}|z|^{2}dt;\\
\\
\displaystyle\int_{-T}^{T}\langle Az,z\rangle dt =\int_{-T}^{T}u^{2}|z|^{2}dt\\
\\
\displaystyle P^{-}(Az-u^{2}z)=0
\end{array}
\right.
$$

\noindent
Since we are studying the behavior of such solutions when $T\to \infty$, it is important to investigate the dependence of this critical point on $T$. Therefore, we proceed by rescaling the interval to $[-1,1]$. The new energy functional, then, reads as follow
$$E(u,z)=\frac{T}{2}\left(\int_{-1}^{1}\frac{1}{T^{2}}|u'(s)|^{2}+\frac{1}{4}u^{2}(s)\ ds+\int_{-1}^{1}\langle A_{\frac{1}{T}}z,z\rangle(s) \ ds -\int_{-1}^{1}u^{2}|z|^{2}\ ds\right)$$
where $A_{\frac{1}{T}}z=-\frac{1}{T}Jz'+JBz$. Setting $\varepsilon=\frac{1}{T}$, we define
$$E_{\varepsilon}(u,z)=\frac{1}{2\varepsilon}\left(\int_{-1}^{1}\varepsilon^{2}|u'(s)|^{2}+\frac{1}{4}u^{2}(s)\ ds+\int_{-1}^{1}\langle A_{\varepsilon}z,z\rangle(s) \ ds -\int_{-1}^{1}u^{2}|z|^{2}\ ds\right).$$
The critical points of $E_{\varepsilon}$ correspond to to solutions to the system
\begin{equation}\label{epseq}
\left\{\begin{array}{ll}
-\varepsilon^{2}u''+\frac{1}{4}u=u|z|^{2}\\
& \text{ on } [-1,1]\\
-\varepsilon J z' +JBz=u^{2}z .
\end{array}
\right.
\end{equation}

\noindent
We will also use the following rescaled norms which are adapted to our problem:
$$\|u\|^{2}_{1,\varepsilon}=\frac{1}{\varepsilon}\int_{[-1,1]}\varepsilon^{2}|u'|^{2}+\frac{1}{4}u^{2}\ dt, \quad
\|v\|^{2}_{\frac{1}{2},\varepsilon}=\frac{1}{\varepsilon}\int_{[-1,1]} \left(|A_{\varepsilon}|^{\frac{1}{2}}|z|\right)^{2}\ dt ,$$
$$\|(u,z) \|^2_{\varepsilon}=\|u\|^2_{1, \varepsilon} + \|z\|^2_{\frac{1}{2},\varepsilon},$$
and finally
$\|u\|_{L^{p},\varepsilon}^{p}=\displaystyle \frac{1}{\varepsilon}\int_{[-1.1]}|u|^{p}\ dt$, for $1\leq p< \infty$.\\
From now on, we will say that a sequence $(u_{\varepsilon},z_{\varepsilon})\in H^{1}_{per}(1)\times H^{\frac{1}{2}}_{per}(1)$ satisfies property $(\mathcal{A})$, if there exist $0<c_{1}<c_{2}$ such that
$$\qquad\qquad c_{1}\leq E_{\varepsilon}(u_{\varepsilon},z_{\varepsilon})\leq c_{2} \quad \text{ and } \quad \|\nabla E_{\varepsilon}(u_{\varepsilon},z_{\varepsilon})\|_{\varepsilon}\to 0 \qquad\qquad (\mathcal{A})$$

\begin{proposition} \label{propbound}
Let $(u_{\varepsilon},z_{\varepsilon})\in H^{1}_{per}(1)\times H^{\frac{1}{2}}_{per}(1)$ satisfying $(\mathcal{A})$. Then:
\begin{itemize}
\item[(i)] $\|u_{\varepsilon}\|_{1,\varepsilon}$ and $\|z_{\varepsilon}\|_{\frac{1}{2},\varepsilon}$ are bounded.
\item[(ii)] $\|z_{\varepsilon}^{-}-g(u_{\varepsilon},z_{\varepsilon}^{+})\|_{\frac{1}{2},\varepsilon} \to 0$.
\item[(iii)] $\|\nabla F_{\varepsilon}(u_{\varepsilon},z_{\varepsilon}^{+})\|_{\varepsilon}\to 0$.
\end{itemize}
\end{proposition}

\begin{proof}
The first point is similar to the proof of the (PS) condition in Lemma \ref{EP-S}, so we omit it. We focus on the second and last point. We set
$$g_{\varepsilon}=g(u_{\varepsilon},z_{\varepsilon}^{+}), \quad z_{1}=z_{\varepsilon}^{+}+g_{\varepsilon}, \quad z_{2}=z_{\varepsilon}^{-}-g_{\varepsilon}$$
so that $z_{\varepsilon}=z_{1}+z_{2}$ and $z_{2}\in H^{\frac{1}{2},-}$. Then we recall that $\langle \nabla_{z}E_{\varepsilon}(u,z_{1}),z_{2}\rangle=0$. Hence,
$$-\langle g_{\varepsilon},z_{2}\rangle -\frac{1}{\varepsilon}\int_{[-1,1]}u_{\varepsilon}^{2}\langle z_{1},z_{2}\rangle \ dt =0.$$
On the other hand, since $\|\nabla E_{\varepsilon}(u_{\varepsilon},z_{\varepsilon})\|_{\varepsilon}\to 0$, we have
$$\langle \nabla_{z}E_{\varepsilon}(u_{\varepsilon},z_{\varepsilon}),z_{2}\rangle= -\langle z_{\varepsilon}^{-},z_{2}\rangle -\frac{1}{\varepsilon}\int_{[-1,1]}u_{\varepsilon}^{2}\langle z_{\varepsilon},z_{2}\rangle \ dt=o(\|z_{2}\|_{\frac{1}{2},\varepsilon}).$$
By taking the difference, it leads to
$$\|z_{2}\|^{2}_{\frac{1}{2},\varepsilon}+\frac{1}{\varepsilon}\int_{[-1,1]}u_{\varepsilon}^{2}|z_{2}|^{2}\ dt =o(\|z_{2}\|_{\frac{1}{2},\varepsilon}).$$
Thus
$$\|z_{2}\|_{\frac{1}{2},\varepsilon}\leq o(\|\nabla E_{\varepsilon}(u_{\varepsilon},z_{\varepsilon})\|_{\varepsilon})=o(1),$$
which proves $(ii)$.\\
For the proof of $(iii)$, we start by writing
$$\nabla F_{\varepsilon}(u_{\varepsilon},z^{+}_{\varepsilon})=\nabla E_{\varepsilon}(u_{\varepsilon},z_{1})=\nabla E_{\varepsilon}(u_{\varepsilon},z_{\varepsilon}-z_{2}).$$
Expanding the last term and using $\nabla E(u_{\varepsilon},z_{\varepsilon})\to 0$ and $z_{2}\to 0$, we have the desired result.
\end{proof}

\begin{lemma}
If $(u_{\varepsilon},z_{\varepsilon})\in H^{1}_{per}(1)\times H^{\frac{1}{2}}_{per}(1)$ satisfies $(\mathcal{A})$, then there exist $t_{\varepsilon}$ and $s_{\varepsilon}$ such that $\Big(t_{\varepsilon}u_{\varepsilon},s_{\varepsilon}z_{\varepsilon}^{+}+g(t_{\varepsilon}u_{\varepsilon},s_{\varepsilon}z_{\varepsilon}^{+})\Big)\in \mathcal{N}$. Moreover, $$(t_{\varepsilon},s_{\varepsilon})\to (1,1), \quad \text{as }\quad \varepsilon \to 0.$$
\end{lemma}

\begin{proof}
Let us consider the map $G:\R \times \R \times H^{\frac{1}{2},-}\to \R \times \R \times H^{\frac{1}{2},-}$ defined by
$$G(t,s,h)=\left(\begin{array}{cc}
\langle \nabla_{u}E_{\varepsilon}(tu_{\varepsilon},sz_{\varepsilon}^{+}+h),tu_{\varepsilon}\rangle\\
\langle \nabla_{z}E(tu_{\varepsilon},s(z_{\varepsilon}^{+}+h)),s(z_{\varepsilon}^{+}+h)\rangle\\
P^{-}\Big(A_{\varepsilon}(z_{\varepsilon}^{+}+h)-t^{2}u_{\varepsilon}^{2}(z_{\varepsilon}^{+}+h)\Big)
\end{array}
\right) .$$
Clearly, $G(t,s,h)=0$ if and only if $(tu_{\varepsilon},s(z_{\varepsilon}^{+}+h))\in \mathcal{N}$. So, we set
$$c_{\varepsilon}=\frac{1}{\varepsilon}\int_{[-1,1]}|u_{\varepsilon}|^{2} \; |z_{\varepsilon}^{+}+g_{\varepsilon}|^{2}\; dt$$
and from condition $(\mathcal{A})$ we can assume that $c_{\varepsilon}\to c_{0}>0$. Indeed, since $(u_{\varepsilon},z_{\varepsilon})$ satisfies $(\mathcal{A})$, we have from Proposition \ref{propbound} that  $\|u_{\varepsilon}\|_{1,\varepsilon}$ and $\|z_{\varepsilon}\|_{\frac{1}{2},\varepsilon}$ are bounded. Moreover,
$$\langle \nabla_{u} E_{\varepsilon}(u_{\varepsilon},z_{\varepsilon}),u_{\varepsilon}\rangle=o(1) \text{ and } \langle \nabla_{z} E_{\varepsilon}(u_{\varepsilon},z_{\varepsilon}),z_{\varepsilon}\rangle =o(1).$$
Hence
$$0<c_{1} \leq E_{\varepsilon}(u_{\varepsilon},z_{\varepsilon})=\frac{1}{2\varepsilon}\int_{[-1,1]}|u_{\varepsilon}|^{2}|z_{\varepsilon}|^{2}\; dt+o(1).$$
On the other hand, again using Proposition \ref{propbound}, we have $\|z_{\varepsilon}^{-}-g(u_{\varepsilon},z_{\varepsilon}^{+})\|_{\frac{1}{2},\varepsilon} \to 0$.
Therefore, we have
$$c_{\varepsilon}=\frac{1}{\varepsilon}\int_{[-1,1]}|u_{\varepsilon}|^{2}|z_{\varepsilon}|^{2}\; dt+o(1)\geq 2c_{1}+o(1).$$
We then compute
$$K=DG(1,1,g(u_{\varepsilon},z_{\varepsilon}^{+}))=\left[\begin{array}{lcc}
B_{\varepsilon}& C_{1}\\
C_{1}^{*}& \widetilde{A}
\end{array}
\right],$$


\noindent
where we have denoted
$$\widetilde{A}\varphi=P^{-}(A_{\varepsilon}\varphi-|u_{\varepsilon}|^{2}\varphi),$$ which is an invertible operator on $H^{\frac{1}{2},-}$,
$$B_{\varepsilon}=\left[\begin{array}{lcl}
2\langle \nabla_{u} E_{\varepsilon}(u_{\varepsilon},z^{+}_{\varepsilon}+g_{\varepsilon}),u_{\varepsilon}\rangle & -2c_{\varepsilon}\\
-2c_{\varepsilon}& 2\langle \nabla_{u} E_{\varepsilon}(u_{\varepsilon},z^{+}_{\varepsilon}+g_{\varepsilon})
\end{array}\right],
$$
and finally,
$$C_{1}=\left [\begin{array}{cc}
-2\frac{1}{\varepsilon}\int |u_{\varepsilon}|^{2}\langle z_{\varepsilon}^{+}+g_{\varepsilon}, \cdot\rangle\ dt\\
0
\end{array}
\right].$$

\noindent
Notice that $$B_{\varepsilon}\to B_{0}:=\left[\begin{array}{lcc}
0 & -2c_{0}\\
-2c_{0}& 0
\end{array}\right], \text{ as } \varepsilon \to 0.$$
Moreover, since $B_{0}$ is invertible and $ C_{1}^{*}B_{0}^{-1}C_{1}=0$, we have that $K$ is invertible for $\varepsilon$ small enough and $K^{-1}$ is bounded uniformly as $\varepsilon \to 0$. Hence, by the inverse function theorem, since $G(1,1,g_{\varepsilon})\to 0$ as $\varepsilon$ goes to zero, there exists $\varepsilon_{0}>0$ such that for all $\varepsilon \in (0,\varepsilon_{0})$,  there exists $t_{\varepsilon}$ and $s_{\varepsilon}$ so that
$$G(t_{\varepsilon},s_{\varepsilon},h)=0.$$
Moreover, one easily sees that $|t_{\varepsilon}-1|+|s_{\varepsilon}-1|\leq O(\|\nabla E_{\varepsilon}(u_{\varepsilon}, z_{\varepsilon})\|_{\varepsilon})$.
\end{proof}

\begin{lemma}\label{lem2}
If $(u_{\varepsilon},z_{\varepsilon})\in H^{1}_{per}(1)\times H^{\frac{1}{2}}_{per}(1)$ satisfies $(\mathcal{A})$, then there exists $(\tilde{u}_{\varepsilon},\tilde{z}_{\varepsilon})\in \mathcal{N}$ such that
$$E_{\varepsilon}(\tilde{u}_{\varepsilon},\tilde{z}_{\varepsilon})=E_{\varepsilon}(u_{\varepsilon},z_{\varepsilon})+o(1).$$
\end{lemma}

\begin{proof}
The proof here is straightforward and it follows directly from the previous Lemma. Indeed, Let $\tilde{u}_{\varepsilon}=t_{\varepsilon}u_{\varepsilon}$ and $\tilde{z}_{\varepsilon}=s_{\varepsilon}(z_{\varepsilon}^{+}+g(t_{\varepsilon}u_{\varepsilon},z_{\varepsilon}^{+}))$. Then we have
\begin{align}
E_{\varepsilon}(\tilde{u}_{\varepsilon},\tilde{z}_{\varepsilon})&=E_{\varepsilon}\Big(u_{\varepsilon}+(t_{\varepsilon}-1)u_{\varepsilon},z_{\varepsilon}+(s_{\varepsilon}-1)z_{\varepsilon}^{+}+ g(t_{\varepsilon}u_{\varepsilon},z_{\varepsilon}^{+})-g(t_{\varepsilon}u_{\varepsilon},z_{\varepsilon}^{+})\notag\\
&\quad+g(t_{\varepsilon}u_{\varepsilon},z_{\varepsilon}^{+})-z_{\varepsilon}^{-}+(s_{\varepsilon}-1)g(t_{\varepsilon}u_{\varepsilon},z_{\varepsilon}^{+})\Big)\notag\\
&=E_{\varepsilon}(u_{\varepsilon},z_{\varepsilon})+O(\|\nabla E_{\varepsilon} (u_{\varepsilon},z_{\varepsilon})\|^{2}_{\varepsilon})\notag
\end{align}
\end{proof}

\noindent
It is important to notice that if $(u_{\varepsilon},z_{\varepsilon})$ is the solution obtained from the min-max process (or minimization on $\mathcal{N}$), then there exists $c_{0}>0$ such that
\begin{equation}\label{lower}
\frac{1}{\varepsilon}\int_{[-1,1]}u_{\varepsilon}^{2}|z_{\varepsilon}|^{2}\ dt \geq c_{0}.
\end{equation}
Indeed, we have
\begin{align}
E_{\varepsilon}(u_{\varepsilon},z_{\varepsilon})&\geq \sup_{t>0,s>0,w\in H^{\frac{1}{2},-}}E(tu_{\varepsilon},sz_{\varepsilon}^{+}+w)\notag\\
&\geq \sup_{t>0,s>0}E(tu_{\varepsilon},sz_{\varepsilon}^{+})\notag\\
&\geq \max_{t>0,s>0}\Big( t^{2}\|u_{\varepsilon}\|^{2}_{1,\varepsilon}+s^{2}\|z^{+}_{\frac{1}{2},\varepsilon}\|^{2}_{\varepsilon}-\frac{1}{2}\Big(t^{4}\|u_{\varepsilon}\|^{4}_{L^{4},\varepsilon}+s^{4}\|z_{\varepsilon}^{+}\|_{L^{4},\varepsilon}^{4}\Big)\Big)\notag\\
&\geq \max_{t>0,s>0}\Big( C_{1}t^{2}-C_{2}t^{4}+\tilde{C}_{1}s^{2}-\tilde{C}_{2}s^{4} \Big) \geq c_{0},\notag
\end{align}
where $C_{1}, \tilde{C}_{1}, C_{2}, \tilde{C}_{2}$ are constants that depend on $c_{1}$ and $c_{2}$ appearing in condition $(\mathcal{A})$. Thus, if we define $\delta_{\varepsilon}$ by $$\delta_{\varepsilon}=\inf_{(u,z)\in \mathcal{N}}E_{\varepsilon}(u,z),$$
then
\begin{equation}\label{ineq}
E_{\varepsilon}(u_{\varepsilon},z_{\varepsilon})\geq \delta_{\varepsilon}\geq c_{0}>0.
\end{equation}
\smallskip

\noindent
Now we want to find an upper bound for $\delta_{\varepsilon}$, and in order to do that we need to construct a suitable sequence $(u_{\varepsilon},z_{\varepsilon})$ satisfying $(\mathcal{A})$ and computationally friendly. We consider then the limiting functional defined on $H^{1}(\R; \R)\times H^{\frac{1}{2}}(\R; \C)$ by
$$E(u,z)=\frac{1}{2}\Big( \int_{\R} |u'|^{2}+\frac{1}{4}u^{2} + \langle Az,z\rangle -u^{2}|z|^{2}\ dt\Big).$$
Its critical points satisfy the following Euler-Lagrange equation
\begin{equation}\label{eqlim}
\left\{\begin{array}{lll}
-u''+\frac{1}{4}u=u|z|^{2}\\
& \text{ on } \R\\
Az=u^{2}z.
\end{array}
\right.
\end{equation}

\noindent
We denote by $\mathcal{M}$ the set of ground state solutions of $(\ref{eqlim})$ and we let
$$\delta_{0}=\inf\Big\{E(u,z); \nabla E(u,z)=0\Big\}=E(U,Z),$$
for $(U,Z)\in \mathcal{M}$.

\begin{lemma}\label{classificationUPSI}
Let $(U,Z)\in \mathcal{M}$. Then up to translation and scaling,
\begin{equation}\label{UPSI}
U(t)=2^{-\frac{1}{4}}\cosh^{-\frac{1}{2}}(t), \qquad
Z(t)=\frac{3}{2\sqrt{2}}\cosh^{-\frac{3}{2}}(t)\left(\begin{array}{ll}
e^{-\frac{t}{2}} \\
e^{\frac{t}{2}}
\end{array}
\right) . 
\end{equation}
\end{lemma}

\begin{proof}
We recall from \cite{BM} that all the ground state solutions of $(\ref{eq1})$ with $M=\S^{3}$ or $\R^{3}$ are classified. Indeed, if $(U,\Psi)$ is a ground state solution, then there exists a parallel spinor $\Phi_{0}\in \S_{\C}^{2}$, $x_{0}\in \R^{3}$ and $\lambda >0$ such that
\begin{align}\label{UPSIR3}
  U(x) & =U_{\lambda}(x)=\left(\frac{2\lambda}{\lambda^{2}+|x-x_{0}|^{2}}\right)^{\frac{1}{2}} \notag\\
  \Psi(x) & =\Psi_\lambda(x)=\left(\frac{2\lambda}{\lambda^{2}+|x-x_{0}|^{2}}\right)^{\frac{3}{2}}\big(\mathbf{1}- (x-x_{0})\big)\cdot\Phi_{0},
\end{align}
where $\mathbf{1}$ denotes the identity endomorphism of the spinor bundle. In particular, if $x_{0}=0$, we see that $U$ is radial and $\Psi \in E(\R^{3})$, as defined in (\ref{ER3}). So any ground state solution satisfies our radial ansatz. Hence, after the change to cylindrical coordinates, we obtain the expression in (\ref{UPSI}) and we have that $(U,\Psi)\in \mathcal{M}$. Therefore the energy level $\delta_{0}$ corresponds indeed to ground state solutions of $(\ref{eq1})$ on $\R^{3}$, which finishes the proof.
\end{proof}

\noindent
\begin{lemma}\label{lem3}
Let $(U,Z)\in \mathcal{M}$ and $\beta \in C^{\infty}_{c}(-1,1)$ such that $\beta=1$ on $\left[-\frac{1}{2},\frac{1}{2}\right]$. We set
$$\overline{u}_{\varepsilon}(t)=\beta(t) U\left(\frac{ t}{\varepsilon}\right), \qquad \overline{z}_{\varepsilon}(t)=\beta(t)Z\left(\frac{ t}{\varepsilon}\right). $$
Then we have
$$E_{\varepsilon}(\overline{u}_{\varepsilon},\overline{z}_{\varepsilon})\to \delta_{0} \quad \text{  and } \quad  \nabla E_{\varepsilon}(\overline{u}_{\varepsilon},\overline{z}_{\varepsilon})\to 0$$ 
as $\varepsilon \to 0.$
\end{lemma}

\begin{proof}
First, we observe that
\begin{align}
\nabla_{u}E_{\varepsilon}(\overline{u}_{\varepsilon},\overline{z}_{\varepsilon}) &=  -\beta(t)U''\left(\frac{t}{\varepsilon}\right)-2\varepsilon\beta'(t)U'\left(\frac{t}{\varepsilon}\right)-\varepsilon^{2}\beta''(t)U\left(\frac{t}{\varepsilon}\right)+ \notag\\
&\quad + \frac{1}{4}\beta(t)U\left(\frac{t}{\varepsilon}\right)-\beta^{3}(t)U\left(\frac{t}{\varepsilon}\right)\left|Z\left(\frac{t}{\varepsilon}\right)\right|^{2}\notag\\
&=2\varepsilon\beta'(t)U'\left(\frac{t}{\varepsilon}\right)+\varepsilon^{2}\beta''(t)U\left(\frac{t}{\varepsilon}\right)+ \left(\beta(t)-\beta^{3}(t)\right)U\left(\frac{t}{\varepsilon}\right)\left|Z\left(\frac{t}{\varepsilon}\right)\right|^{2}.\notag
\end{align}

\noindent
Next we notice that
$$\frac{1}{\varepsilon}\int_{[-1,1]}\varepsilon^{2}\left|\beta'(t)U'\left(\frac{t}{\varepsilon}\right)\right|^{2}\ dt =\varepsilon^{2}\int_{\left[-\frac{1}{\varepsilon},\frac{1}{\varepsilon}\right]}|\beta'(\varepsilon t))U'(t)|^{2}\ dt\leq C\varepsilon^{2}\int_{\R}|U'|^{2}\ dt \to 0.$$
Similarly,
$$\frac{1}{\varepsilon}\int_{[-1,1]}\varepsilon^{4}\Big|\beta''(t)U\left(\frac{t}{\varepsilon}\right)\Big|^{2}\ dt\leq C\varepsilon^{4}\int_{\R}|U(t)|^{2}\ dt \to 0.$$
But for the last term, we have
$$\frac{1}{\varepsilon}\int_{[-1,1]}(\beta(t)-\beta^{3}(t))^{2}\left|U\left(\frac{t}{\varepsilon}\right)\right|^{2} \left|Z\left(\frac{t}{\varepsilon}\right)\right|^{4}\ dt \leq C \int_{\frac{1}{2\varepsilon}\leq |s| \leq \frac{1}{\varepsilon}}U^{2}(s)|Z(s)|^{4}\ ds \to 0.$$
This shows that $\|\nabla_{u}E_{\varepsilon}(\overline{u}_{\varepsilon},\overline{z}_{\varepsilon})\|_{L^{2},\varepsilon} \to 0$ and hence $\nabla_{u}E_{\varepsilon}(\overline{u}_{\varepsilon},\overline{z}_{\varepsilon})\to 0 $ in $H^{-1}_{per}(1)$.

\noindent
Similarly, we can show that $\nabla_{z}E_{\varepsilon}(\overline{u}_{\varepsilon},\overline{z}_{\varepsilon})\to 0$ in $H^{-\frac{1}{2}}$. We move now to the energy part:
\begin{align}
2E_{\varepsilon}(\overline{u}_{\varepsilon},\overline{z}_{\varepsilon})&=\frac{1}{\varepsilon}\int_{[-1,1]}\varepsilon^{2}\left|\beta'(t) U\left(\frac{t}{\varepsilon}\right)+\frac{1}{\varepsilon}\beta(t)U'\left(\frac{t}{\varepsilon}\right)\right|^{2} +\frac{1}{4}\beta^{2}(t)U^{2}\left(\frac{t}{\varepsilon}\right)\notag \\
&\quad + \left\langle \beta(t) J\left(Z'\left(\frac{t}{\varepsilon}\right)+BZ\left(\frac{t}{\varepsilon}\right)\right)+\varepsilon \beta'(t) JZ\left(\frac{t}{\varepsilon}\right),\beta(t) Z\left(\frac{t}{\varepsilon}\right)\right\rangle \notag \\
&\quad -\beta^{4}(t)U^{2}\left(\frac{t}{\varepsilon}\right)\left|Z\left(\frac{t}{\varepsilon}\right)\right|^{2}\ dt\notag\\
&=\int_{[-\frac{1}{\varepsilon},\frac{1}{\varepsilon}]}\beta^{2}(\varepsilon s)\left(|U'(s)|^{2}+\frac{1}{4}U^{2}(s) +\langle AZ,Z\rangle -|U(s)|^{2}|Z(s)|^{2}\right)\ ds \notag \\
&\quad+ \int_{|s|\leq \frac{1}{\varepsilon}}
\varepsilon^{2}|\beta'|^{2}(\varepsilon s)|U(s)|^{2}+2\varepsilon\beta'(\varepsilon s)\beta(\varepsilon s) U'(s)U(s)\ ds \notag \\
&\quad + \int_{|s|\leq \frac{1}{\varepsilon}}\varepsilon \beta'(\varepsilon s)\beta(\varepsilon s)\langle JZ,Z\rangle\ ds \notag \\
&\quad+\int_{\frac{1}{2\varepsilon}\leq |s|\leq \frac{1}{\varepsilon}}\left(\beta^{4}(\varepsilon s)-\beta^{2}(\varepsilon s)\right)|U(s)|^{2} |Z(s)|^{2} \ ds\notag\\
&=I+II+III+IV\notag
\end{align}
By using the dominated convergence theorem, one sees that
$$I\to 2E(U,Z)=2\delta_{0}, \text{ as } \varepsilon\to 0.$$
On the other hand
$$II+III \leq C\varepsilon\to 0.$$
So it remains to show that the last term also converges to zero. Indeed,
$$\int_{\frac{1}{2\varepsilon}\leq |s|\leq
\frac{1}{\varepsilon}}\left(\beta^{4}(\varepsilon s)-\beta^{2}(\varepsilon s)\right)|U(s)|^{2} |Z(s)|^{2} \ ds \leq
C \int_{\frac{1}{2\varepsilon}\leq |s|} |U(s)|^{2}|Z(s)|^{2}\ ds \to 0, $$
as $\varepsilon\to 0$, which finishes the proof of the Lemma.
\end{proof}

\noindent
This previous Lemma shows in particular that $(\overline{u}_{\varepsilon}, \overline{z}_{\varepsilon})$ satisfies $(\mathcal{A})$. In the next Lemma, we provide an upper bound for $ \delta_{\varepsilon}$:
\begin{lemma}
$$\delta_{0} \geq \limsup_{\varepsilon\to 0} \delta_{\varepsilon} .$$
\end{lemma}
\begin{proof}
From Lemma \ref{lem3}, we have that $(\overline{u}_{\varepsilon},\overline{z}_{\varepsilon})$ satisfies assumption $(\mathcal{A})$. Hence, using Lemma \ref{lem2}, we have
$$E_{\varepsilon}(\overline{u}_{\varepsilon},\overline{z}_{\varepsilon})=E_{\varepsilon}(\tilde{u}_{\varepsilon},\tilde{z}_{\varepsilon})+o(1)\geq \delta_{\varepsilon}+o(1).$$
So the conclusion follows by taking the $\limsup$ in both sides.
\end{proof}

\noindent
We also provide a lower bound for $ \delta_{\varepsilon}$:
\begin{lemma}
$$\liminf_{\varepsilon\to 0}\delta_{\varepsilon}\geq \delta_{0}.$$
\end{lemma}

\begin{proof}
Let $(u_{\varepsilon},z_{\varepsilon})$ be a minimizer of $E_{\varepsilon}$ on $\mathcal{N}$. Then, it is a solution to the system $(\ref{epseq})$ and $E_{\varepsilon}(u_{\varepsilon},z_{\varepsilon})=\delta_{\varepsilon}$. We first claim that there exist $r_{0}>0$, $\kappa_{1},\kappa_{2}>0$ and $y_{\varepsilon} \in [-1,1]$, such that
\begin{equation}\label{kappa}
\frac{1}{\varepsilon}\int_{|t-y_{\varepsilon}|\leq \varepsilon r_{0}}|u_{\varepsilon}|^{2} dt >\kappa_{1}, \qquad \frac{1}{\varepsilon}\int_{|t-y_{\varepsilon}|\leq \varepsilon r_{0}}|z_{\varepsilon}|^{2} dt>\kappa_{2}.
\end{equation}
In order to prove this claim, we assume by contradiction that the previous inequalities are false. Without loss of generality, we can assume that the first inequality does not hold; for the second one we argue in the same way. Then, for every $r>0$,

$$\lim_{\varepsilon\to 0} \left(\frac{1}{\varepsilon}\sup_{y\in [-1,1]}\int_{|t-y|\leq 2\varepsilon r}|u_{\varepsilon}|^{2}\ dt \right) =0.$$

\noindent
Now we take $\beta_{\varepsilon,y}$ a  cut-off function in $(y-2r\varepsilon,y+2r\varepsilon)$, such that $\beta_{\varepsilon,y}(t)=1$ for $|t-y|\leq r\varepsilon$. Therefore, we have
$$\lim_{\varepsilon\to 0} \left(\frac{1}{\varepsilon}\sup_{y\in [-1,1]}\int_{[-1,1]}|\beta_{\varepsilon,y}u_{\varepsilon}|^{2}  dt \right)=0.$$
This yields $\|u_{\varepsilon}\|_{L^{q},\varepsilon}\to 0$ for all $2\leq q <\infty$.
Indeed, let $p>q>2$ and $s>0$ so that $q=sp+2(1-s)$. Then we have
\begin{align}
\frac{1}{\varepsilon}\int_{[-1,1]}|\beta_{\varepsilon,y}u_{\varepsilon}|^{q}\ dt &\leq  \left(\frac{1}{\varepsilon}\int_{[-1,1]}|\beta_{\varepsilon,y}u_{\varepsilon}|^{2}\ dt \right)^{1-s} \left(\frac{1}{\varepsilon}\int_{[-1,1]}|\beta_{\varepsilon,y}u_{\varepsilon}|^{p}\ dt\right)^{s}\notag\\
&\leq C\left(\frac{1}{\varepsilon}\int_{[-1,1]}|\beta_{\varepsilon,y}u_{\varepsilon}|^{2}\ dt \right)^{1-s} \|\beta_{\varepsilon,y}u_{\varepsilon}\|_{1,\varepsilon}^{s}.\notag
\end{align}
Hence, if we cover $[-1,1]$ by subintervals of radius $r\varepsilon$ with each subinterval overlapping with at most two others, we get
$$\frac{1}{\varepsilon}\int_{[-1,1]}|u_{\varepsilon}|^{q}\ dt \leq C_{2} \left(\sup_{y\in [-1,1]}\frac{1}{\varepsilon}\int_{|t-y|\leq 2\varepsilon r}|u_{\varepsilon}|^{2}\ dt\right)^{1-s}\|u_{\varepsilon} \|_{1,\varepsilon}^{s}.$$
Thus, one has using $(\ref{lower})$ 
$$
c_{0} \leq \frac{1}{\varepsilon}\int_{[-1,1]}u_{\varepsilon}^{2}|z_{\varepsilon}|^{2}\ dt
\leq \|u_{\varepsilon}\|_{L^{4},\varepsilon}^{2}\|z_{\varepsilon}\|_{L^{4},\varepsilon}^{2}
\leq C\|u_{\varepsilon}\|_{L^{4},\varepsilon}^{2}\|z_{\varepsilon}\|_{\frac{1}{2},\varepsilon}^{2} ,
$$
and the contradiction follows by passing to the limit; hence first claim is proved.\\
Next, we set
$$U_{\varepsilon}(s)=u_{\varepsilon}(\varepsilon s+y_{\varepsilon})\quad \text{ and } \quad Z_{\varepsilon}(s)=z_{\varepsilon}(\varepsilon s+y_{\varepsilon}) .$$
We notice that $(U_{\varepsilon},Z_{\varepsilon})$ is bounded in 
$$\left(H^{1}_{loc}(\R; \R)\times H_{loc}^{\frac{1}{2}}(\R; \C)\right)\cap \left( L^{p}(\R; \R)\times L^{p}(\R; \C)\right)$$
for all $1<p<\infty$. So we can extract a convergent subsequence that converges strongly in $L^{p}$ for all $p$ and weakly in $H^{1}_{loc}\times H^{\frac{1}{2}}_{loc}$ to $(U_{0},Z_{0})$. Therefore, if we take a test function $h\in H^{1}(\R; \R)$ that is compactly supported in $[-R,R]$, we get
$$ \int_{\R}\left[-U''_{\varepsilon}+\frac{1}{4}U_{\varepsilon}-U_{\varepsilon}|Z_{\varepsilon}|^{2}\right]h\ ds =
\frac{1}{\varepsilon}\int_{|t-y_{\varepsilon}|\leq \varepsilon R}\left[-\varepsilon^{2}u_{\varepsilon}''+\frac{1}{4}u_{\varepsilon}-u_{\varepsilon}|z_{\varepsilon}|^{2}\right]\tilde{h}\ dt=0,$$
where $\tilde{h}=h\left(\frac{t-y_{\varepsilon}}{\varepsilon}\right)$. Similarly, taking $\varphi \in H^{\frac{1}{2}}(\R; \C)$, compactly supported in $[-R,R]$, we get
$$\int_{\R} \langle AZ_{\varepsilon}-U^{2}_{\varepsilon}Z_{\varepsilon},\varphi \rangle\ ds=\frac{1}{\varepsilon}\int_{|t-y_{\varepsilon}|\leq \varepsilon R}\langle A_{\varepsilon}z_{\varepsilon}-u_{\varepsilon}^{2}z_{\varepsilon},\tilde{\varphi}\rangle\ dt=0,$$
where $\tilde{\varphi}(t)=\varphi\left(\frac{t-y_{\varepsilon}}{\varepsilon}\right)$. Hence, $(U_{0},Z_{0})$ is a solution of $(\ref{eqlim})$. Moreover, from (\ref{kappa}), we have
$$\int_{[-r_{0},r_{0}]}|U_{\varepsilon}|^{2}\ ds =\frac{1}{\varepsilon}\int_{|t-y_{\varepsilon}|\leq \varepsilon r_{0}}|u_{\varepsilon}|^{2}\ dt>\kappa_{1}>0.$$
and
$$\int_{[-r_{0},r_{0}]}|Z_{\varepsilon}|^{2}\ ds =\frac{1}{\varepsilon}\int_{|t-y_{\varepsilon}|\leq \varepsilon r_{0}}|z_{\varepsilon}|^{2}\ dt>\kappa_{1}>0.$$
Therefore,
$$\int_{[-r_{0},r_{0}]}|U_{0}|^{2}\ ds \not =0 , \qquad  \int_{[-r_{0},r_{0}]} |Z_{0}|^{2} \ ds \not=0.$$
Thus, $(U_{0},Z_{0})$ is not trivial and
$$E(U_{0},Z_{0})\geq \delta_{0}.$$
On the other hand, we have
$$ E_{\varepsilon}(u_{\varepsilon},z_{\varepsilon}) =\frac{1}{2\varepsilon}\int_{|t-y_{\varepsilon}|\leq 1}|u_{\varepsilon}|^{2}|z_{\varepsilon}|^{2}  dt
= \frac{1}{2}\int_{|s|\leq\frac{1}{\varepsilon}}|U_{\varepsilon}|^{2}|Z_{\varepsilon}|^{2} ds $$


\noindent
Therefore,
$$\liminf_{\varepsilon\to 0} \delta_{\varepsilon}=\liminf_{\varepsilon\to 0}E_{\varepsilon}(u_{\varepsilon},z_{\varepsilon})=E(U_{0},Z_{0})\geq \delta_{0}.$$
\end{proof}

\noindent
With the proof of the previous Lemma, we have completed the study of the asymptotic behaviour of the periodic solutions $(u_{\varepsilon},z_{\varepsilon})$ obtained by minimizing $E_{\varepsilon}$ on $\mathcal{N}$. It only remains to prove that these solutions $(u_{\varepsilon},z_{\varepsilon})$ are different from the equilibrium solution for $\varepsilon$ small enough.

\begin{lemma}
There exists $\varepsilon_{0}>0$ such that for $\varepsilon<\varepsilon_{0}$ the ground state solution of $(\ref{epseq})$ is different from the equilibrium solution.
\end{lemma}
\begin{proof}
This is easily verified. Indeed, if
$$u=1, \qquad z=\left(\begin{array}{c}
\pm \frac{1}{2\sqrt{2}}\\
\\
\pm \frac{1}{2\sqrt{2}}
\end{array}
\right),$$
then $$E_{\varepsilon}(u,z)=\frac{1}{2\varepsilon}\int_{[-1,1]}|u|^{2}|z|^{2}\ dt =\frac{1}{4\varepsilon}\to \infty>\delta_{0}.$$
\end{proof}

\noindent
This finish the description of the global picture of the Hamiltonian system, by finding a family of periodic solutions converging to the homoclinic orbit given in (\ref{UPSI}) and thus proving the second part of the main theorem. Moreover, Corollary \ref{Corollary1} follows after a change in cylindrical coordinates, which allows to pass from (\ref{UPSI}) to (\ref{UPSIR3}).

\end{document}